\documentclass[11pt,a4paper]{article}
\usepackage[top=2.5cm, bottom=2.5cm, left=2.5cm, right=2.5cm]{geometry}
\usepackage[latin1]{inputenc}
\usepackage{csquotes}
\usepackage{amsmath}
\usepackage{amsthm}
\usepackage{amsfonts}
\usepackage{amssymb}
\usepackage{graphics}
\usepackage{float}
\usepackage[hang,flushmargin]{footmisc} 

\usepackage{amsmath,amsthm,amssymb,graphicx, multicol, array}
\usepackage{enumerate}
\usepackage{enumitem}
\usepackage{xcolor}

\usepackage[pagebackref,bookmarks,colorlinks,breaklinks]{hyperref}
\hypersetup{linkcolor=blue,citecolor=red,filecolor=blue,urlcolor=blue} 

\usepackage{tabto}

\numberwithin{equation}{section}

\usepackage{tikz}
\usepackage{pgfplots}
\usetikzlibrary{matrix}

\usepackage{mathrsfs}

\DeclareMathOperator{\ord}{ord}

\newtheorem{thm}{Theorem}[section]
\newtheorem{lem}{Lemma}[section]
\newtheorem{conj}{Conjecture}[section]
\newtheorem{exa}{Example}[section]

\newtheorem{cor}{Corollary}[section]
\newtheorem{dfn}{Definition}[section]
\newtheorem{exe}{Exercise}[section]

\newcommand{\N}{\mathbb{N}}
\newcommand{\Z}{\mathbb{Z}}
\newcommand{\Q}{\mathbb{Q}}

\newcommand{\C}{\mathbb{C}}

\newcommand{\tP}{\mathbb{P}}

\usepackage{fancyhdr}
\title{Simultaneous Primitive Roots over Finite Rings}
\date{}
\author{N. A. Carella}

\begin{document}
	
	\thispagestyle{empty}
	\date{}
	\maketitle
	
	\vskip .25 in 
\begin{abstract}
This note investigates the average density of prime numbers $p\in[x,2x]$ with respect to a random simultaneous primitive root $g\leq p^{1/2+\varepsilon}$ over the finite rings $\Z/p\Z$ and $\Z/p^2\Z$ as $x \to \infty$. \let\thefootnote\relax\footnote{ \today \date{} 
}
\end{abstract}

\section{Introduction} \label{S2727SPRFR-I}\hypertarget{S2727SPRFR-I}
Let $x>1$ be a large real number, let $p\in[x,2x]$ be a prime number and let $g\geq2$ be a primitive root modulo $p$. A \textit{stationary} primitive root is a simultaneous primitive root $g$ modulo $p$ and $g$ modulo $p^2$, see \hyperlink{S2727SPRFR.200C}{Definition} \ref{dfn2727SPRFR.200C} for more details. An upper bound for the least stationary primitive root and an asymptotic formula for  the density of prime numbers $p\in[x,2x]$ with respect a random stationary primitive root $g_s\geq2$ are investigated here. \\

\begin{thm} \label{thm2727SPRFR.050}\hypertarget{thm2727SPRFR.050} Let $x>1$ be a large real number and let $p\in[x,2x]$ be a prime number. Then, the least stationary primitive root satisfies the inequality
	\begin{equation} \label{eq2727SPRFR.050i}
		g_s(p)\ll p^{1/2+\varepsilon}\log p, 
	\end{equation} 
	where $\varepsilon>0$ is a small number, as $x\to\infty$. 
\end{thm}	

The corresponding counting function for the number of primes $p\in [x,2x]$ for which there is a primitive root $g\in [2,2z]$ modulo $p$, that can be lifted to a primitive roots modulo $p^2$, is defined by
\begin{eqnarray} \label{eq2727SPRFR.050j}
	N_s(x,z)	&=&\#\{ p \in [x,2x]: \ord_{p}g=\varphi(p),\;\ord_{p^2}g=\varphi(p^2) \text{ and } g\leq 2z\,\}\nonumber\\[.3cm]
	&=&	\#\{ \;\Psi_{s} (g): x\leq p\leq 2x \text{ and } g\leq 2z\;\},
\end{eqnarray}	
where  $z=p^{1/2+\varepsilon}\log p$ and $\varepsilon>0$ is a small number.

\begin{thm} \label{thm2727SPRFR.100}\hypertarget{thm2727SPRFR.100} Let $x>1$ be a large real number, let $p\in[x,2x]$ be a prime number and let $z=p^{1/2+\varepsilon}\log p$. Then, on average, a primitive root $g\leq 2z$ modulo $p$ is stationary for a positive proportion of the primes $p\in[x,2x]$. Specifically,
	\begin{equation} \label{eq2727SPRFR.100j}
		\frac{N_s(x,z)}{z^2}=c_2\cdot\frac{x}{\log x}\left( 1+O\left( \frac{1}{\log x}  \right)  \right), 
	\end{equation} 
	where $0<c_2<0.260653  $ is a small constant,	as $x\to\infty$. 
\end{thm}	

A primitive root $g\in \Z/p\Z$ that fails to be a primitive root $g\in \Z/p^2\Z$ is called a \textit{nonstationary primitive root}. In this case the corresponding counting function for the number of primes $p\in [x,2x]$ for which there is a primitive root $g\in [2,2z]$ modulo $p$, that cannot be lifted to a primitive roots modulo $p^2$, is defined by
\begin{eqnarray} \label{eq2727SPRFR.050ia}
	N_n(x,z)	&=&\#\{ p \in [x,2x]: \ord_{p}g=\varphi(p),\;\ord_{p^2}g\ne\varphi(p^2) \text{ and } g\leq 2z\,\}\nonumber\\[.3cm]
	&=&	\#\{ \;\Psi_{n} (g): x\leq p\leq 2x \text{ and } g\leq 2z\;\}.
\end{eqnarray}	 
\begin{thm} \label{thm2727SPRFR.100N}\hypertarget{thm2727SPRFR.100N} Let $x>1$ be a large real number, let $p\in[x,2x]$ be a prime number and let $z=p^{1/2+\varepsilon}\log p$. Then, on average, a primitive root $g\leq 2z$ modulo $p$ is nonstationary for a positive proportion of the primes $p\in[x,2x]$. Specifically,
	\begin{equation} \label{eq2727SPRFR.100ib}
		\frac{N_n(x,z)}{z^2}=c_3\cdot\frac{x}{\log x}\left( 1+O\left( \frac{1}{\log x}  \right)  \right), 
	\end{equation} 
	where $0<c_3 <0.113302$ is a small constant,	as $x\to\infty$. 
\end{thm}	
These results are closely related to the problem discussed in \cite{PA2009} regarding the density of prime numbers with the same stationary least primitive root. In particular, if $g(p)$ and $h(p)$ denote the least primitive roots modulo $p$ and modulo $p^2$ respectively, then the asymptotic formulas
\begin{equation} \label{eq2727SPRFR.100s}
	N_L(x,z)	=\#\{ \;p \in [x,2x]: g(p)=h(p) \text{ and } g\leq 2z\,\}
\end{equation}	
and 
\begin{equation} \label{eq2727SPRFR.100t}
	\overline{N_L}(x,z)	=\#\{ \;p \in [x,2x]: g(p)\ne h(p) \text{ and } g\leq 2z\,\}
\end{equation}	
or for any specific stationary primitive root $g$ remain as open problems.\\

A survey on various results and the history of the expansions of the rational numbers $1/n\in \Q$ appears in \cite{BM2009}. The previous result has an application to this repeated decimal problems. 

\begin{cor} \label{cor2727SPRFR.100}\hypertarget{cor2727SPRFR.100} Let $x>1$ be a large real number, let $p\in[x,2x]$ be a prime number and let $\ell\leq p^{1/2+\varepsilon}\log p$ be an integer $\ne\pm1, v^2$. Then, the followings statements hold.

\begin{enumerate}[font=\normalfont, label=(\roman*)]
\item On average, the $\ell$-adic expansion of the fraction 
\begin{equation} \label{eq2727SPRFR.100t2}
1/p^2=0.\overline{a_1a_2a_3\cdots a_t}
\end{equation}	
has maximal period $t=p(p-1)$ for a positive proportion $c_2>0$ of the primes $p\in[x,2x]$ as $x\to\infty$. 
\item On average, the $\ell$-adic expansion of the fraction 
\begin{equation} \label{eq2727SPRFR.100t3}
1/p^k=0.\overline{a_1a_2a_3\cdots a_t}
\end{equation}	 
has maximal period $t=p^{k-1}(p-1)$ for a positive proportion $c_2>0$ of the primes $p\in[x,2x]$ for all $k\geq2$ as $x\to\infty$.
\end{enumerate}
\end{cor}	

The first few sections covers the notation, definitions and basic foundation. The proof of  
\hyperlink{thm2727SPRFR.050}{Theorem} \ref{thm2727SPRFR.050} appears in \hypertarget{S2727SPRFR-SPR}{Section} \ref{S2727SPRFR-SPR}. The proof of \hyperlink{thm2727SPRFR.100}{Theorem} \ref{thm2727SPRFR.100} appears in \hypertarget{S2727SPRFR-DP}{Section} \ref{S2727SPRFR-DP}. Except for the characteristic function, compare \hyperlink{lem2727SPRFR.100-C4}{Lemma} \ref{lem2727SPRFR.100-C4} and \hyperlink{lem2727SPRFR.100-C8}{Lemma} \ref{lem2727SPRFR.100-C8}, the proof of  
\hyperlink{thm2727SPRFR.100N}{Theorem} \ref{thm2727SPRFR.100N} is essentially the same as the proof of \hyperlink{thm2727SPRFR.100}{Theorem} \ref{thm2727SPRFR.100}, it is omitted. The proof of  
\hyperlink{cor2727SPRFR.100}{Corollary} \ref{cor2727SPRFR.100} is a corollary of \hyperlink{thm2727SPRFR.900}{Theorem} \ref{thm2727SPRFR.900}, which is completed in \hypertarget{S2727SPRFR-DP}{Section} \ref{S2727SPRFR-DP}.\\

Let $\tP=\{2,3,5,\ldots\}$ be the set of primes, let $\pi(x)=\#\{\,p\leq x: p \in \tP\,\}$ and let $g\ne\pm1,v^2$ be an integer. The above result provides the density on average, it is expected that this can be conditionally proven, using similar analytic methods as in the case $k=1$, see \cite{HC1967}, for any fixed primitive root $g\ne\pm1,v^2$. But the specific density 
\begin{equation} \label{eq2727SPRFR.100i}
	\delta_s(g)	=\lim_{x\to\infty}\frac{\#\{\, p \in \tP: \ord_{p}g=\varphi(p),\;\ord_{p^2}g=\varphi(p^2)\,\}}{\pi(x)}
\end{equation}	
of prime numbers $p\geq2$ with respect to a fixed integer $g\ne\pm1, v^2$ remains as an open problem in algebraic number theory, this is related to the simpler case for $\Z/p\Z$, see \cite{SP2003}, \cite{LM2014}.

\section{Notation, Definitions and Basic Concepts} \label{S2727SPRFR-N}\hypertarget{S2727SPRFR-N}
The sets $\N=\{0,1,2,3,\ldots\}$ and $\Z=\{\ldots, -3,-2,-1, 0,1,2,3,\ldots\}$ denote the sets of natural numbers and the sets of integers respectively.

\begin{dfn}\label{dfn2727SPRFR.200C}\hypertarget{S2727SPRFR.200C} {\normalfont Let $ p\geq3$ be a prime and let $\Z/p^k\Z$ be the finite ring of integers modulo $p^k$. The multiplicative order of an element $u\in \Z/p^k\Z$ is defined by
$$\ord_{p^k}u=\min\{n:u^n\equiv 1\bmod p^k\}.$$		
An element of maximal order $\ord_{p^k}u=\varphi(p^k)$ is called a \textit{primitive root} modulo  $p^k$.}
\end{dfn}

\begin{dfn}\label{dfn2727SPRFR.200F}\hypertarget{S2727SPRFR.200F} {\normalfont Let $ p\geq3$ be a prime. A primitive root in the finite ring $\Z/p\Z$ is called \textit{stationary} if it remains a primitive root in every finite ring its multiplicative order $\Z/p^k\Z$ and $\Z/2p^k\Z$ independently of $k\geq1$. Otherwise, it is called a \textit{nonstationary} primitive root.}
\end{dfn}

\begin{lem}\label{lem9696.500FR-FE}\hypertarget{lem9696.500FR-FE}  {\normalfont (Fermat-Euler)} If \(a\in \mathbb{Z}\) is an integer such that \(\gcd (a,n)=1,\) then \(a^{\varphi (n)}\equiv
	1 \bmod n\).
\end{lem}
\begin{lem} \label{lem9696.500FR}\hypertarget{lem9696.500FR}  {\normalfont (Primitive root test)} An integer $u \in \Z$ is a primitive root modulo an integer $n \in \N$ if and only if 
	\begin{equation}\label{eq9696.152}
		u^{\varphi (n)/p} -1\not \equiv 0 \mod  n
	\end{equation}
	for all prime divisors $p \mid \varphi (n)$.
\end{lem}
The primitive root test is a special case of the Lucas primality test, introduced in {\color{red}\cite[p.\ 302]{ LE1878}}. A more recent version appears in {\color{red}\cite[Theorem 4.1.1]{CP2005}}, and similar sources. 
\section{Lifting Results} \label{S9977FRL}\hypertarget{S9977FRL}
A simple primitive root lifting procedure $\Z/p\Z\longrightarrow \Z/p^k\Z$, with $ k\geq2$, is also applicable to the larger finite ring $\Z/2p^k\Z$. The version for the larger finite ring is presented here.

\begin{lem}\label{lem9977FRL.400G}\hypertarget{lem9977FRL.400G} {\normalfont (Primitive root lift test)} A nonzero integer $u \in \Z/2p^k\Z$ is a primitive root modulo $2p^k$ if and only if it is not a $(p-1)$th root of unity in $\left( \Z/2p^k\Z\right)^{\times}$. Specifically,
	\begin{equation}\label{eq9955FRL.400Gc}
		u^{p-1}\not \equiv 1 \mod  2p^k\nonumber.
	\end{equation}
\end{lem}

\begin{proof}[\textbf{Proof}] For any odd prime $p>2$ and any integer $k\geq2$, the prime divisors $q$ of $\varphi (2p^k)=p^{k-1}(p-1)$ are the same as the prime divisors $q$ of $p(p-1)$. Furthermore, the primitive root test, \hyperlink{lem9696.500FR}{Lemma} \ref{lem9696.500FR}, states that $u\in \left( \Z2p^k\Z\right) ^{\times}$ is a primitive root if and only if 
	\begin{equation}\label{eq9977FRL.400Gj}
		u^{\frac{p(p-1)}{q}} -1\not \equiv 0 \mod  2p^k
	\end{equation} 
	is true for all prime $q\mid \varphi (2p^k)=p^{k-1}(p-1)$. Since $\tau \in\Z/p\Z$ is already a primitive root, it follows that \eqref{eq9977FRL.400Gj} is true for all prime divisors $q\mid p-1$. 
\end{proof}

\begin{lem}\label{lem9977FRL.400T}\hypertarget{lem9977FRL.400T} Let $p$ be an odd prime and let $\tau \in \Z/p\Z$ be the set of primitive roots. Then
	
	\begin{enumerate}[font=\normalfont, label=(\roman*)]
		\item If $\tau$ is a primitive root in $ \Z/p\Z$ then $ \tau \text{ or } \tau+p $
		is a primitive root in $\displaystyle \Z/p^k\Z$ for every integer $k\geq1$.
		\item If $\tau$ is a primitive root in $ \Z/p\Z$ then $ \tau \text{ or } \tau^{-1} $
		is a primitive root in $\displaystyle \Z/p^k\Z$ for every integer $k\geq1$, where $\tau^{-1} \bmod p$ is the inverse.
	\end{enumerate}	
\end{lem}
\begin{proof}[\textbf{Proof}] Try it as an exercise, for example apply the \textit{primitive root test} described above or consult the literature, say {\color{red}\cite[Lemma 1.4.5]{CH1996}} for (i). 
\end{proof}

Another very simple procedure generates the primitive roots of an adjacent finite ring. For large parameters $p>2$ and $k\geq2$, the lifting procedures described below are more efficient than the procedure that generates the primitive roots by exponentiations. 
\begin{lem}\label{lem9977FRL.400D}\hypertarget{lem9977FRL.400D}Let $p$ be an odd prime and let $\mathscr{T}_k=\{\tau_k :\ord_{p^k}\tau_k=p^{k-1}(p-1)\}$ be the set of primitive roots in $\Z/p^k\Z$. If $k\geq1$, then
	\begin{enumerate}[font=\normalfont, label=(\roman*)]
		\item Every primitive root in $\displaystyle \left( \Z/p^{k+1}\Z\right)^{\times}$ is of the form $\displaystyle \tau_{k+1}=\tau_k+ap^{k}$,
		where $a\in \Z/p\Z$ and $\tau_k\in\mathscr{T}_k$.
		\item Every primitive root in $\displaystyle \left( \Z/2p^{k+1}\Z\right)^{\times}$ is of the form $\displaystyle \tau_{k+1}=\tau_k+ap^{k}$,
		where $a\in \Z/p\Z$ and $\tau_k\in\mathscr{T}_k$.
	\end{enumerate}
\end{lem}
\begin{proof}[\textbf{Proof}] There are several methods available to verify these results, see \hyperlink{lem9977LPFR.300}{Lemma} \ref{lem9977LPFR.300} for a proof based on the primitive root test, \hyperlink{lem9696.500FR}{Lemma} \ref{lem9696.500FR}, and binomial expansions. 
\end{proof}
A very similar procedure generates the primitive roots of a nonadjacent finite ring. 
\begin{lem}\label{lem9977FRL.400B}\hypertarget{lem9977FRL.400B}Let $p$ be an odd prime and let $\mathscr{T}=\{\tau :\ord_p\tau=p-1\}$ be the set of primitive roots in $\Z/p\Z$. If $k\geq2$, then
	\begin{enumerate}[font=\normalfont, label=(\roman*)]
		\item Every primitive root in $\displaystyle \left( \Z/p^k\Z\right)^{\times}$ is of the form $\displaystyle \tau_k=\tau+ap^{k-1}$,
		where $a\in \Z/p^{k-1}\Z$ and $\tau\in\mathscr{T}$.
		\item Every primitive root in $\displaystyle \left( \Z/2p^k\Z\right)^{\times}$ is of the form $\displaystyle \tau_k=\tau+ap^{k-1}$,
		where $a\in \Z/2p^{k-1}\Z$ and $\tau\in\mathscr{T}$.
	\end{enumerate}
\end{lem}

The first part of \hyperlink{lem9977FRL.400T}{Lemma} \ref{lem9977FRL.400T} is a special case of this result.
\section{Primitive Roots in the Finite Ring $\Z/p^2\Z$} \label{S9977LPFR}\hypertarget{S9977LPFR}
The relationship between the first two finite rings $\Z/p\Z$ and $\Z/p^2\Z$ seems to be the most important ones --- the relationships between any other two finite rings $\Z/p^k\Z$ and $\Z/p^m\Z$, with $k\ne m$ are simple variations of the former.\\

Let the symbols $g(p)$ and $h(p)$ denote the least primitive root in $\Z/p\Z$ and $\Z/p^2\Z$ respectively. For primes in the range $\leq 10^{12}$, only two cases have been found to have distinct least primitive roots, for example, a primitive root in $\Z/p\Z$ does not lift to a primitive root in $\Z/p^2\Z$. The largest example in \autoref{table100-2} was very recently discovered, see \cite{PA2009}. \\

\vskip .1 in 
\begin{table}[H]
	\setlength{\tabcolsep}{0.5cm}
	\renewcommand{\arraystretch}{1.50092}
	\setlength{\arrayrulewidth}{0.75pt}
	\centering
	\begin{tabular}{l|l|l}
		
		Prime $p$	& $g(p)\in \Z/p\Z$& $h(p)\in \Z/p^2\Z$\\
		\hline
		$p<10^{12}$&$g(p)\geq2$&$h(p)=g(p)$\\
		\hline
		$40487$&$g(p)=5$&$h(p)=10$\\
		\hline
		$6692367337$&$g(p)=5$&$h(p)=7$
	\end{tabular}  
	\vskip .1 in		
	\caption{Least primitive root in $\Z/p\Z$ and $\Z/p^2\Z$}
	\label{table100-2}
\end{table}

\begin{lem}\label{lem9977LPFR.300}\hypertarget{lem9977LPFR.300} For any odd prime $p$, a primitive root $\tau \in \Z/p\Z$ lifts to a primitive root 
	$\tau \in \Z/p^2\Z$ if and only if 
	$$\tau^{p-1}\not \equiv 1 \bmod p^2.$$
\end{lem}
\begin{proof}[\textbf{Proof}] The primitive root test in \hyperlink{lem9977FRL.400G}{Lemma} \ref{lem9977FRL.400G} states that an element $\tau \in \Z/p^2\Z$ is a primitive root if and only if 
	\begin{equation}\label{eq9977LPFR.300d}
		\tau^{\varphi(p^2)/q}\not\equiv 1 \bmod p^2
	\end{equation}
	for all prime divisor $q\mid \varphi(p^2)=p(p-1)$. This test is verified in 2 steps.\\
	
	Case (i) $q\mid p-1$. The information $\pm1\ne \tau\in \left( \Z/p\Z\right)^{\times} $ is a primitive root implies that $\tau^{(p-1)/q}\not \equiv \pm1\bmod p$. Rewriting $\tau^{(p-1)/q}$ as an integer then reducing mod $p^2$ yield   
	\begin{align}\label{eq9977LPFR.300da}
		\tau^{(p-1)/q}&\ne \pm1+a_1p+a_2p^2+\cdots+a_kp^k\\[.3cm]
		&\not \equiv \pm1\bmod p^2\nonumber,	
	\end{align}
	where $0\leq|a_i|<p$ and some $k\geq0$. Consequently, 
	\begin{align}\label{eq9977LPFR.300db}
		\tau^{p(p-1)/q}&\ne \left( \pm1+a_1p+a_2p^2+\cdots+a_kp^k\right)^{p(q-1/q)} \\[.3cm]
		&\not \equiv \left( \pm1\right)^{p} \bmod p^2\nonumber\\[.3cm]
		&\not \equiv \pm1\bmod p^2\nonumber.
	\end{align}
	Case (ii) $q=p$. Rewriting $\tau^{p(p-1)/q}$ as an integer and reducing mod $p^2$ lead to the relation 
	\begin{align}\label{eq9977LPFR.300dc}
		\tau^{p(p-1)/q}&=\left( a_0+a_1p+a_2p^2+\cdots+a_kp^k\right)^{p-1} \\[.3cm]
		& =a_0^{p-1}+(p-1)a_0^{(p-2)}a_1p+\cdots +a_k^{p-1}p^{k(p-1)} \nonumber\\[.3cm]
		& \equiv 1+(p-1)a_0^{(p-2)}a_1p \bmod p^2\nonumber\\[.3cm]
		& \equiv 1-a_0^{(p-2)}a_1p \bmod p^2\nonumber,
	\end{align}
	where $0\leq|a_i|<p$ and some $k\geq0$. Since $a_1$ is unknown, this congruence remains undetermined.
\end{proof}
For some primes $p$ and some primitive root $\tau \bmod p$ there are possibilities that the values $a_1=0$, but it seems to be a very difficult problem to determine which primitive roots modulo $p$ cannot be lifted modulo $p^2$. This topic is related to the Wieferich primes problem, see \cite{PA2009}, \cite{KJ2023}, et alii.\\

To develop a formula for the exceptional value $a_1$, assume the left side is 1 and set $a_0=\tau$ in the second line of \eqref{eq9977LPFR.300dc}. Solving it yields,
\begin{equation}\label{eq9977LPFR.300f0}
	a_1\equiv 	\frac{ 1-\tau^{p-1}}{p}\cdot \left( (p-1)\tau^{p-2}\right) ^{-1} \bmod p^2,
\end{equation}
where $\left( (p-1)\tau^{p-2}\right) ^{-1}\equiv a \bmod p$ is the inverse modulo $p$.
A different approach to the proof of \hyperlink{lem9977LPFR.300}{Lemma} \ref{lem9977LPFR.300} is given in {\color{red}\cite[Lemma 2]{AT1976}}.\\

\begin{exa}\label{exa9977LPFR.200b}{\normalfont A demonstration of \hyperlink{lem9977LPFR.300}{Lemma} \ref{lem9977LPFR.300}. Let $p=41$ be a prime number. Define the subsets 
		\begin{align}\label{eq9977LPFR.300f1}
			\mathscr{T}(p)&=\{\text{ primitive roots } \bmod p\}\\
			&=\{6,7,11,12,13,15,17,19,22,24,26,28,29,30,34,35\}\nonumber 
		\end{align}		
		and
		\begin{align}\label{eq9977LPFR.300f2}
			\mathscr{T}(p^2)&=\{\text{ primitive roots } \bmod p^2\}\\
			&=\{6,7,11,12,13,15,17,19,22,24,26,28,29,30,34,35,\ldots\}\nonumber 
		\end{align}			
		The numerical data in \eqref{eq9977LPFR.300f1} and \eqref{eq9977LPFR.300f2} show that the intersection $\mathscr{T}(41)\cap \mathscr{T}(41^2)=\mathscr{T}(41)$. So, every primitive root in $\mathscr{T}(p)$ lifts to a primitive root in $\mathscr{T}(p^2)$.	
		
	}		
\end{exa}

\begin{exa}\label{exa9977LPFR.200c}{\normalfont A demonstration of \hyperlink{lem9977LPFR.300}{Lemma} \ref{lem9977LPFR.300}. Let $p=43$ be a prime number. Define the subsets 
		\begin{align}\label{eq9977LPFR.300c1}
			\mathscr{T}(p)&=\{\text{ primitive roots } \bmod p\}\\
			&=\{3,5,12,18,19,20,26,28,29,30,33,34\}\nonumber 
		\end{align}		
		and
		\begin{align}\label{eq9977LPFR.300c2}
			\mathscr{T}(p^2)&=\{\text{ primitive roots } \bmod p^2\}\\
			&=\{3,5,12,18,20,26,28,29,30,33,34,46,48,55,61,62,63,\ldots\}\nonumber 
		\end{align}			
		
		The numerical data in \eqref{eq9977LPFR.300c1} and \eqref{eq9977LPFR.300c2} show that the intersection $\mathscr{T}(43)\cap \mathscr{T}(43^2)\ne\mathscr{T}(43)$. So, every primitive root in $\mathscr{T}(p)$, but $g=19$, lifts to a primitive root in $\mathscr{T}(p^2)$.	In this case $19^{42}\equiv 1\bmod 43$, so $19\in \mathscr{T}(p)$ but $19\not\in \mathscr{T}(p^2)$. However, $19+43=62\in \mathscr{T}(p^2)$ is a primitive root mod $p^2$ --- this follows from \hyperlink{lem9977FRL.400T}{Lemma} \ref{lem9977FRL.400T}.
	}		
\end{exa}

\begin{exa}\label{exa9977LPFR.200d}{\normalfont A demonstration of \hyperlink{lem9977LPFR.300}{Lemma} \ref{lem9977LPFR.300} for the first prime $p=40487$ that has a least primitive root modulo $p$ that can not be lifted to a primitive root modulo $p^2$. Define the subsets 
		\begin{align}
			\mathscr{T}(p)&=\{\text{ primitive roots } \bmod p\}\\
			&=\{5,10,13,15,17,26,29,30,34,35,38,39,40,\ldots\}\nonumber 
		\end{align}		
		and
		\begin{align}
			\mathscr{T}(p^2)&=\{\text{ primitive roots } \bmod p^2\}\\
			&=\{10,13,15,17,26,29,30,34,35,38,39,40,\ldots\}\nonumber 
		\end{align}	
		The first round of congruences (d), (e) and (f) proves that $5$ is primitive root modulo $p$. 
		\begin{multicols}{2}
			\begin{enumerate}[font=\normalfont, label=(\alph*)]
				\item $p=40487$,
				\item $p-1=2\cdot 31\cdot653$,
				\item $\varphi(p)=p-1$,
				\item	$\displaystyle 	5^{31\cdot653}\equiv -1 \bmod 40487$,
				\item $\displaystyle 	5^{2\cdot653}\equiv 32940 \bmod 40487$,
				\item$\displaystyle 	5^{2\cdot31}\equiv 4413 \bmod 40487$
			\end{enumerate}		
		\end{multicols}		
		On the second round of congruences below
		\begin{multicols}{2}
			\begin{enumerate}[font=\normalfont, label=(\alph*)]
				\item $p^2=1,639,197,169$,
				\item $p(p-1)=2\cdot 31\cdot653\cdot40487$,
				\item $\varphi(p^2)=p(p-1)$,
				\item	$\displaystyle 	5^{31\cdot653}\equiv -1 \bmod 40487^2$,
				\item $\displaystyle 	5^{2\cdot653}\equiv 1089538110 \bmod 40487^2$,
				\item$\displaystyle 	5^{2\cdot31}\equiv 1388749000 \bmod 40487^2$,
				\item$\displaystyle 	5^{2\cdot 31\cdot653}\equiv 1 \bmod 40487^2$,
			\end{enumerate}		
		\end{multicols}	
		congruence (g) proves that $5$ is not a primitive root modulo $p^2$. However, $5+40487=40492\in \mathscr{T}(p^2)$ is a primitive root mod $p^2$ --- this follows from \hyperlink{lem9977FRL.400T}{Lemma} \ref{lem9977FRL.400T}.		
	}		
\end{exa}


\section{Estimates of Omega Function} \label{S4545EOFW-W}\hypertarget{S4545EOFW-W}
The little omega function $\omega:\N\longrightarrow \N$ defined by $\omega(n)=\#\{p\mid n:\text{ prime }p\}$ is a well studied function in number theory. It has the explicit upper bound $\omega(n)\leq 2\log n/\log\log n$, see {\color{red}\cite[Proposition 7.10]{DL2012}} {\color{red}\cite[Theorem 2.10]{MV2007}} and \cite{EP1985}. It emerges in the finite sum
\begin{equation} \label{eq4545EOFW400j}
	\sum _{d\mid n}|\mu(n)| 
	=2^{\omega(n)}		
	\ll n^{\varepsilon},
\end{equation}
where $\varepsilon>$ is an arbitrarily small number, which is frequently used in finite field analysis. The corresponding estimates for the big omega function are also important in finite field analysis.

\

For $u\ne0$, an integer $n\geq1$ and a multiplicative character $\chi\ne1$, there is a trivial upper bound for the exponential sum
\begin{eqnarray} \label{eq2727EES.400i}
	\sum _{1<d\mid n} \frac{\mu(d)}{\varphi(d)}\sum _{\ord \chi =d} \chi(u) 
	&\leq &\sum _{1<d\mid n} \left |\frac{\mu(d)}{\varphi(d)}\right|\left |\sum _{\ord \chi =d} \chi(u)\right| \\[.3cm]
	&\ll &\sum _{1<d\mid n}|\mu(n)| \nonumber\\[.3cm]		
	&\ll&n^{\varepsilon},
	\nonumber
\end{eqnarray}
  The last line in \eqref{eq2727EES.400i} follows from 

where The average order of the widely used exponential sum \eqref{eq2727EES.400i} over the integers is significantly smaller. In fact, it is 
\begin{equation} \label{eq4545EOFW.400h}
	x^{-1}\sum_{n\leq x}\sum _{1<d\mid n}|\mu(n)| \ll \log x.
\end{equation} 
This follows from
\begin{equation} \label{eq4545EOFW.400l}
	\sum _{n\leq x}	2^{\omega(n)}=c_0x\log x+c_1x+O(x^{1/2}\log x)		,
\end{equation}
where $c_0,c_1>0$ are constants, see {\color{red}\cite[p.\; 42]{MV2007}}. On the other hand, the average order over the shifted primes should be of the form
\begin{equation} \label{eq4545EOFW.400m}
	\sum _{p\leq x}	2^{\omega(p-1)}\overset{?}{=}c_2\frac{x}{\log x}\log \log x+c_3\frac{x}{\log x}+O\left( \frac{x}{(\log x )^2}\right),
\end{equation}
where $c_2,c_3>0$ are constants, this is not in the literature.\\

This is similar to the order of the divisor function, which has a nearly explicit upper bound of the form $\sum_{d\mid n}1=n^{(\log 2+o(1)/\log\log n)}$, see {\color{red}\cite[Proposition 7.12]{DL2012}}, {\color{red}\cite[Theorem 315]{HW1979}} and similar sources.

\section{Characteristic Functions in Finite Rings} \label{S2727CFFR-F}\hypertarget{S2727CFFR-F}
A representation of the characteristic function dependent on the orders of the cyclic groups is given below. This representation is sensitive to the primes decompositions $q=p_1^{e_1}p_2^{e_2}\cdots p_t^{e_t}$, with $p_i$ prime and $e_i\geq1$, of the orders of the cyclic groups $q=\# G$. 

\begin{lem} \label{lem2727CFFR.100-F}\hypertarget{lem2727CFFR.100-F}
	Let \(G\) be a finite cyclic group of order \(q=\# G\), and let \(0\neq u\in G\) be an invertible element of the group. Then
	\begin{equation}\label{eq2727CFFR.100d}
		\Psi (u)=\frac{\varphi (q)}{q-1}\sum _{d \mid q} \frac{\mu (d)}{\varphi (d)}\sum _{\ord(\chi ) = d} \chi (u)=
		\left \{\begin{array}{ll}
			1 & \text{ if } \ord_q (u)=\varphi (q),  \\
			0 & \text{ if } \ord_q (u)\neq \varphi (q). \\
		\end{array} \right .
	\end{equation}
\end{lem}

The authors in \cite{DH1937}, \cite{WR2001} attribute this formula to Vinogradov, and other authors attribute this formula to Landau, \cite{LE1927}. The proof and other details on the characteristic function are given in {\color{red}\cite[p. 863]{ES1957}}, {\color{red}\cite[p.\ 258]{LN1997}}, {\color{red}\cite[p.\ 18]{MP2007}}. The characteristic function for multiple primitive roots is used in {\color{red}\cite[p.\ 146]{CZ1998}} to study consecutive primitive roots. In \cite{DS2012} it is used to study the gap between primitive roots with respect to the Hamming metric. And in \cite{WR2001} it is used to prove the existence of primitive roots in certain small subsets \(A\subset \mathbb{F}_p\). In \cite{DH1937} it is used to prove that some finite fields do not have primitive roots of the form $a\tau+b$, with $\tau$ primitive and $a,b \in \mathbb{F}_p$ constants. In addition, the Artin primitive root conjecture for polynomials over finite fields was proved in \cite{PS1995} using this formula.

\section{Simultaneous Characteristic Functions } \label{S2727SPRFR-C}\hypertarget{S2727SPRFR-C}
Let $g(p)\geq2$ be the least primitive roots modulo $p$ and let $h(p)\geq2$ be the least primitive roots modulo $p^2$. The characteristic functions of the primitive roots in the finite rings $\Z/p\Z$ and $\Z/p^2\Z$ are defined by

\begin{equation}\label{eq2727SPRFR.100d}
	\Psi_{p} (g)=
	\begin{cases}
		1&\text{ if } \ord_pg=p-1,\\
		0&\text{ if } \ord_pg\ne p-1,
	\end{cases}
\end{equation}
and 
\begin{equation}\label{eq2727SPRFR.100f}
	\Psi_{p^2} (g)=
	\begin{cases}
		1&\text{ if } \ord_{p^2}g=p(p-1),\\
		0&\text{ if } \ord_{p^2}g\ne p(p-1),
	\end{cases}
\end{equation}
respectively. The basic foundation of the indicators functions \eqref{eq2727SPRFR.100d} and \eqref{eq2727SPRFR.100f}, in terms of exponential functions, are fully described in \hyperlink{S2727CFFR-F}{Section} \ref{S2727CFFR-F}.

\begin{lem} \label{lem2727SPRFR.100-C4}\hypertarget{lem2727SPRFR.100-C4} Let $p>1$ be a prime number and let $g\in \left( \Z/p\Z\right) ^{\times}$. Then, the characteristic function of stationary primitive roots in the finite rings $\Z/p\Z$ and  $\Z/p^2\Z$ is defined by
\begin{equation}\label{eq2727SPRFR.100h}
	\Psi_{s} (g)=	\frac{\Psi_{p} (g)\left( 1+\Psi_{p^2} (g)\right)}{2}=
	\begin{cases}
		0&\text{ if } \ord_{p}g\ne p-1\text{ and } \ord_{p^2}g\ne p(p-1),\\
		0&\text{ if } \ord_{p}g\ne p-1\text{ and } \ord_{p^2}g= p(p-1),\\
		0&\text{ if } \ord_{p}g=p-1\text{ and } \ord_{p^2}g\ne p(p-1),\\
		1&\text{ if } \ord_{p}g=p-1\text{ and } \ord_{p^2}g=p(p-1).\nonumber
	\end{cases}
\end{equation}
\end{lem}

\begin{proof}[\textbf{Proof}] A shifted product of these indicator functions \eqref{eq2727SPRFR.100d} and \eqref{eq2727SPRFR.100f} produces the indicator function \eqref{eq2727SPRFR.100h} of a restricted collection of primitive roots. This is precisely the indicator function for the collection of stationary primitive roots in the finite rings $\Z/p\Z$, that can be lifted to primitive roots in the finite rings $\Z/p^2\Z$. 
\end{proof}
The same technique provides a characteristic function for nonstationary primitive roots over finite rings. 
\begin{lem} \label{lem2727SPRFR.100-C8}\hypertarget{lem2727SPRFR.100-C8} Let $p>1$ be a prime number and let $g\in \left( \Z/p\Z\right) ^{\times}$. Then, the characteristic function of nonstationary primitive roots in the finite rings $\Z/p\Z$ and  $\Z/p^2\Z$ is defined by
	\begin{equation}\label{eq2727SPRFR.100k}
\Psi_{n} (g)=	\frac{\Psi_{p} (g)\left( 1-\Psi_{p^2} (g)\right)}{2}=
		\begin{cases}
			0&\text{ if } \ord_{p}g\ne p-1\text{ and } \ord_{p^2}g\ne p(p-1),\\
			0&\text{ if } \ord_{p}g\ne p-1\text{ and } \ord_{p^2}g= p(p-1),\\
			1&\text{ if } \ord_{p}g=p-1\text{ and } \ord_{p^2}g\ne p(p-1),\\
			0&\text{ if } \ord_{p}g=p-1\text{ and } \ord_{p^2}g=p(p-1).\nonumber
		\end{cases}
	\end{equation}
\end{lem}

\section{Results for the Totient Function over Shifted Primes}\label{S4040DP-TF}\hypertarget{S4040DP-TF}
The average order of the totient function ratio $\varphi(n)/n$ over the shifted primes $n=p-1$ emerges on various results in the theory of primitive roots. The asymptotic formula of the first case $k=1$ of the result below was proved decades ago in \cite{SP1969}. 
\begin{lem} \label{lem4040DP.200A}\hypertarget{lem4040DP.200A} {\normalfont {({\color{red}\cite[Lemma 4.4]{VR1973}}) }} Let $x\geq 1$ be a large number, and let $\varphi (n)$ be the Euler totient function. If $k\geq1$ is an integer, then
	\begin{equation}\label{eq4040.200d}
		\sum_{p\leq x }\left( \frac{\varphi(p-1)}{p-1}\right) ^k  =a_k\frac{x}{\log x} +
		O\left(\frac{x}{(\log x)^2}\right) ,
	\end{equation}
where the constant
\begin{equation}\label{eq4040DP.200fk}
	a_k=\prod_{p\geq2}\left(1-\frac{p^k-(p-1)^k}{p^k(p-1)} \right), 
\end{equation}	
as $x \to \infty$. 
\end{lem}

The constants of interest in this application are
\begin{equation} \label{eq4040DP.200hk}
	a_1=\prod_{p \geq 2 } \left(1-\frac{1}{p(p-1)}\right)=0.373956099060845279979647798266673361\ldots, 
\end{equation}  
and 
\begin{equation} \label{eq4040DP.200kl}
a_2=\prod_{p \geq 2 } \left(1-\frac{2p-1}{p^2(p-1)}\right)<  0.1473496249460471189049141150422354
\ldots,
\end{equation} 
the numerical approximations are based on the first $10^4$ primes. The first constant $a_1>0$ is known as Artin constant, it is the average density of primes with respect to a random primitive roots modulo $p$, but the second one $a_2>0$ is not well known.  \\

It is clear that for each prime $p\geq2$, the local factors satisfy 

\begin{equation} \label{eq4040DP.200k1}
 1-\frac{2p-1}{p^2(p-1)}<1-\frac{1}{p(p-1)}.
\end{equation} 
In the current applications, there two constants of interest. These are 
\begin{equation} \label{eq4040DP.200k2}
	c_2=\frac{a_1+a_2}{2}<  0.26065286200344619944228095665445442967965\ldots,
\end{equation}
and 
\begin{equation} \label{eq4040DP.200k3}
	c_3=\frac{a_1-a_2}{2}<0.11330323705739908053736684161221893192939\ldots.
\end{equation}

\section{Estimates of Exponential Sums} \label{S2727EES-A}\hypertarget{S2727EES-A}
An estimate of the exponential sum \eqref{eq2727EES.450d} over the set of prime numbers is proved in {\color{red}\cite[Section 2]{ES1957}}. The general version over the set of integers is based on the generalized Gaussian sum
\begin{lem}   \label{lem2727EES.450E}\hypertarget{lem2727EES.450E} Let $N\geq1$ be a large integer and let $U,V\subset \N $ be a pair of subsets of integers of cardinalities $\#U,\#V\leq N$. If $\chi\ne1 $ is a multiplicative character modulo $N$ and \(u+v\ne0\), then
	\begin{equation} \label{eq2727EES.450d}
		\sum _{u\in U,v\in V} \chi(u+v) 
		\ll N^{1/2} \cdot \sqrt{\#U\cdot \#V }.
	\end{equation} 
\end{lem} 

\begin{proof}[\textbf{Proof}] Let $\chi\ne1$ be a nontrivial multiplicative character mod $N$ and let $\tau(\chi)=	\sum _{1 \leq k\leq N}\chi(k)e^{i2\pi k/N}$. The product $\tau(\overline{\chi})\chi(t)=	\sum _{1 \leq k\leq N}\overline{\chi}(k)e^{i2\pi kt/N}$, where $\overline{\chi}$ is the complex conjugate. Summing the product $\tau(\overline{\chi})\chi(t)$ over the subsets $U$ and $V$ yields
	\begin{eqnarray} \label{eq2727EES.450i}
		\tau(\overline{\chi})	\sum _{u\in U,v\in V} \chi(u+v) 
		&=&\sum _{0 \leq k< N}\overline{\chi}(k)\sum _{u\in U,v\in V} e^{i2\pi k(u+v)/N}\\[.3cm]
		&=&\sum _{0 \leq k< N}\overline{\chi}(k)\sum _{u\in U} e^{i2\pi ku/N}\sum _{v\in V} e^{i2\pi kv/N}\nonumber.
	\end{eqnarray} 
	Taking absolute values followed by an application of the Schwarz inequality yield
	\begin{eqnarray}\label{eq2727EES.450j}
		|\tau(\overline{\chi})|^2\left|	\sum _{u\in U,v\in V} \chi(u+v) \right|^2
		&\leq&\sum _{0\leq k< N}\left|\sum _{u\in U} e^{i2\pi ku/N}\right|^2\sum _{0 \leq k< N}\left|\sum _{v\in V} e^{i2\pi kv/N}\right|^2.
	\end{eqnarray}
	The first term on the right side has the upper bound
	\begin{eqnarray} \label{eq2727EES.450k}
		\sum _{0 \leq k< N}\left|\sum _{u\in U} e^{i2\pi ku/N}\right|^2
		&=&\sum _{0 \leq k< N}\sum _{\substack{u_0\in U\\u_1\in U}} e^{i2\pi k(u_0-u_1)/N}\\[.3cm]
		&=&\sum _{\substack{u_0\in U\\u_1\in U}} \sum _{0 \leq k<N}e^{i2\pi k(u_0-u_1)/N}\nonumber\\[.3cm]
		&=&N\cdot \#U\nonumber.
	\end{eqnarray}
	The procedure as in \eqref{eq2727EES.450k} yields
	\begin{equation} \label{eq2727EES.450m}
		\sum _{0 \leq k< N}\left|\sum _{v\in V} e^{i2\pi kv/N}\right|^2=N\cdot \#V\nonumber.
	\end{equation}
	Replacing \eqref{eq2727EES.450k} and \eqref{eq2727EES.450m} into \eqref{eq2727EES.450j}
	completes the verification.
\end{proof}

\begin{lem}   \label{lem2727EES.450A}\hypertarget{lem2727EES.900A} Let $N\geq1$ be a large integer and let $U,V\subset \N $ be a pair of subsets of integers of cardinalities $\#U,\#V\leq N$. If $\chi\ne1 $ is an additive character modulo $N$ and \(uv\ne0\), then
	\begin{equation} \label{eq2727EES.450mr}
		\sum _{u\in U,v\in V} \psi(uv) 
		\ll N^{1/2} \cdot \sqrt{\#U\cdot \#V }.
	\end{equation} 
\end{lem} 

\begin{proof}[\textbf{Proof}] The proof is similar to the previous case.\end{proof}
\section{Main Term} \label{S2727SPRFR-M}\hypertarget{S2727SPRFR-M}
The main term in \eqref{eq2727SPRFR.900m4} corresponds to the finite sum attached to the trivial character $\chi=1$. Using the ratio $\varphi(n)/n\gg (\log\log n)^{-1}>0$ for all integers $n\gg1$, see {\color{red}\cite[Proposition 8.4]{DL2012}}, it easy to compute the lower bound
\begin{equation}
M(x,z_0,z_1)\gg \frac{x z_0 z_1}{\log x \log \log x}. 
\end{equation}
However, a more precise asymptotic formula is computed here. 

\begin{lem} \label{lem2727SPRFR.300-M}\hypertarget{lem2727SPRFR.300-M} Let $x>1$ be a large real number and let $\#U=z_0$ and $\#V=z_1$ be the cardinalities of the subsets of integers $U$ and $V$. Then 
\begin{eqnarray}\label{eq2727SPRFR.300f}
M(x,z_0,z_1)&=&\frac{1}{2}\sum _{x\leq p \leq 2x} \frac{\varphi(p-1)}{p-1}\left( 1+\frac{\varphi(\varphi(p^2))}{p^2}\right) \sum _{ u\in U, v\in V,}  1\nonumber\\[.3cm]
		&=& c_2z_0 z_1\cdot\frac{x}{\log x}\left( 1+O\left(\frac{1 }{\log x}\right)  \right) \nonumber,
	\end{eqnarray}
where $c_2=(a_1+a_2)/2\leq 0.481840$ is the density constant.
\end{lem}

\begin{proof}[\textbf{Proof}] Rearrange it as two separate: 
\begin{eqnarray}\label{eq2727SPRFR.300i}
	2M(x,z_0,z_1)&=& \sum _{x\leq p \leq 2x} \frac{\varphi(p-1)}{p-1}\left( 1+\frac{\varphi(\varphi(p^2))}{p^2}\right) \sum _{ u\in U, v\in V,}  1\\[.3cm]
&=&\sum _{ u\in U, v\in V,}  \sum _{x\leq p \leq 2x} \frac{\varphi(p-1)}{p-1}+\sum _{ u\in U, v\in V,}  \sum _{x\leq p \leq 2x} \frac{\varphi(p-1)}{p-1}\cdot\frac{\varphi(\varphi(p^2))}{p^2}\nonumber\\[.3cm]	
&=&M_0(x,z_0,z_1)\;+\;M_1(x,z_0,z_1)\nonumber.
\end{eqnarray} 
Set $z_0=\#U$ and $z_1=\#V$. Take $k=1$, by \hyperlink{lem4040DP.200A}{Lemma} \ref{lem4040DP.200A} the first subsum
\begin{eqnarray}\label{eq2727SPRFR.300i2}
M_0(x,z_0,z_1)
&=&\sum _{ u\in U, v\in V,}  \sum _{x\leq p \leq 2x} \frac{\varphi(p-1)}{p-1}\\[.3cm]	
&=& z_0 z_1\cdot\sum _{x\leq p \leq 2x} \frac{\varphi(p-1)}{p-1}\nonumber\\[.3cm]	
&=&a_1z_0 z_1\cdot\left(\frac{2x}{\log 2x}-\frac{x}{\log x}+O\left(\frac{x }{(\log x)^2}\right)  \right) \nonumber\\[.3cm]	
&=&a_1z_0z_1\cdot\frac{x}{\log x}\left( 1+O\left(\frac{x }{\log x}\right)  \right) \nonumber.
\end{eqnarray}

Take $k=2$, by \hyperlink{lem4040DP.200A}{Lemma} \ref{lem4040DP.200A} the second subsum 
\begin{eqnarray}\label{eq2727SPRFR.300i4}
	M_1(x,z_0,z_1)
	&=&\sum _{ u\in U, v\in V,}  \sum _{x\leq p \leq 2x} \frac{\varphi(p-1)}{p-1}\cdot\frac{\varphi(\varphi(p^2))}{p^2}\\[.3cm]	
	&=&\sum _{ u\in U, v\in V,}  \sum _{x\leq p \leq 2x} \left(  \frac{\varphi(p-1)}{p-1}\right) ^2+O\left( \sum _{ u\in U, v\in V,}  \sum _{x\leq p \leq 2x} \frac{1}{p}  \right)  \nonumber\\[.3cm]	
	&=&a_2z_0 z_1\cdot\left(\frac{2x}{\log 2x}-\frac{x}{\log x}+O\left(\frac{x }{(\log x)^2}\right)  \right) +O\left(\frac{z_0 z_1 }{\log x}\right)\nonumber\\[.3cm]	
	&=&a_2z_0 z_1\cdot\frac{x}{\log x}\left( 1+O\left(\frac{1 }{\log x}\right)  \right) \nonumber.
\end{eqnarray} 
The second line in \eqref{eq2727SPRFR.300i4} follows from the simplification
\begin{equation}\label{eq2727SPRFR.300i6}
	\frac{\varphi(p-1)}{p-1}\cdot \frac{\varphi(\varphi(p^2))}{p^2}=\left( \frac{\varphi(p-1)}{p-1}\right) ^2+O\left( \frac{1}{p}\right). 
\end{equation} 
Summing \eqref{eq2727SPRFR.300i2} and\eqref{eq2727SPRFR.300i4} completes the verification.
\end{proof}

\section{Estimate of the Error Term} \label{S2727PRFRNL-E}\hypertarget{S2727PRFRNL-E}
The error term $E(x,z)$ in \eqref{eq2727SPRFR.900m4} corresponds to the finite sum attached to the nontrivial characters $\chi\ne1$. An upper bound of this term is computed in this section.

\begin{lem}  \label{lem2727SPRFR.550-E}\hypertarget{lem2727SPRFR.550-E} Let $x>1$ be a large real number and let $\#U=z_0$ and $\#V=z_1$ be the cardinalities of the subsets of integers $U$ and $V$. If $\chi_1\ne1$ and $\chi_2\ne1$ are multiplicative characters of orders $d_1\mid \varphi(p)$ and $d_2\mid \varphi(p^2)$ respectively, then 
\begin{eqnarray} \label{eq2727SPRFR.550d}
E(x,z)
&=&\sum _{ u\in U, v\in V,}  \sum _{x\leq p \leq 2x} \left (\frac{\varphi(p-1)}{p-1}\sum _{\substack{d_1\mid \varphi(p)\\d_1>1}} \frac{\mu(d_1)}{\varphi(d_1)}\sum _{\ord \chi_1 =d_1} \chi_1(u) \right ) \\[.3cm]
	&&\hskip 1 in +\sum _{ u\in U, v\in V,}  \sum _{x\leq p \leq 2x} \left (\frac{\varphi(p-1)}{p-1}\sum _{\substack{d_1\mid \varphi(p)\\d_1>1}} \frac{\mu(d_1)}{\varphi(d_1)}\sum _{\ord \chi_1 =d_1} \chi_1(u) \right )      \nonumber\\[.3cm]
	&&\hskip 1 in \times \left (\frac{\varphi(\varphi(p^2))}{p^2}\sum _{\substack{d_2\mid \varphi(p^2)\\d_2>1}} \frac{\mu(d_2)}{\varphi(d_2)}\sum _{\ord \chi_2 =d_2} \chi_2(u) \right ) \nonumber\\[.3cm]
&\ll& \frac{x\cdot \sqrt{z_0 z_1 }}{\log x} \cdot p^{1/2+\varepsilon} \nonumber,
\end{eqnarray}
where $\varepsilon>0$ is a small real number.
\end{lem}
\begin{proof}[\textbf{Proof}] The error term is estimated as a sum $E(x,z)= E_{0}(x,z)\;+\;E_{1}(x,z)$ of two simpler suberror terms.\\

	The first suberror term $E_0(x)$ is estimated in \hyperlink{lem2727SPRFR.550B}{Lemma} \ref{lem2727SPRFR.550B} and the second suberror term $E_1(x)$ is estimated in \hyperlink{lem2727SPRFR.550C}{Lemma} \ref{lem2727SPRFR.550C}.  Summing these estimates yields
	\begin{eqnarray} \label{eq2727SPRFR.550f}
		E(x)&=& E_{0}(x)\;+\;E_{1}(x)   \\[.3cm]
		&\ll& \frac{x\cdot \sqrt{z_0 z_1 }}{\log x} \cdot p^{1/2+\varepsilon} \;+\; \frac{xz_0 z_1}{(\log x)(\log \log x)^2} \cdot 2^{\omega(p-1)+\omega(p^2)}p^{1/2}\nonumber\\[.3cm]
		&\ll& \frac{x}{\log x} \cdot p^{1/2+\varepsilon}\cdot \sqrt{z_0 z_1 }\nonumber,
	\end{eqnarray}
	where $\varepsilon>0$ is a small real number. This completes the estimate of the error term .
\end{proof}

\begin{lem}  \label{lem2727SPRFR.550B}\hypertarget{lem2727SPRFR.550B} Let $x>1$ be a large real number and let $\#U=z_0$ and $\#V=z_1$ be the cardinalities of the subsets of integers $U,V\subset \Z$. If the functions $\chi_1\ne1$ are multiplicative characters of orders $d\mid \varphi(p)$, then 
	\begin{eqnarray} \label{eq2727SPRFR.400ms}
		E_0(x,z_0,z_1)
		&=&\sum _{ u\in U, v\in V,}  \sum _{x\leq p \leq 2x} \left (\frac{\varphi(p-1)}{p-1}\sum _{\substack{d_1\mid \varphi(p)\\d_1>1}} \frac{\mu(d_1)}{\varphi(d_1)}\sum _{\ord \chi_1 =d_1} \chi_1(u+v) \right ) \nonumber\\[.3cm]
				&\ll&\frac{x\cdot \sqrt{z_0 z_1 }}{\log x} \cdot p^{1/2+\varepsilon}  ,
	\end{eqnarray}
where $\varepsilon>0$ is a small real number. 	 
\end{lem}
\begin{proof}[\textbf{Proof}]Taking absolute values and applying \hyperlink{lem2727EES.900E}{Lemma} \ref{lem2727EES.450E} yield
\begin{eqnarray} \label{eq2727SPRFR.400m2}
|E_0(x,z_0,z_1)|
&\leq&\sum _{x\leq p \leq 2x} \frac{\varphi(p-1)}{p-1}\sum _{\substack{d_1\mid \varphi(p)\\d_1>1}}\left| \frac{\mu(d_1)}{\varphi(d_1)}\right |\sum _{\ord \chi_1 =d_1} \left| \sum _{ u\in U, v\in V,} \chi_1(u+v) \right |\nonumber \\[.3cm]
&\leq&\sum _{x\leq p \leq 2x} 1\sum _{d_1\mid \varphi(p)}\left| \mu(d_1)\right |\cdot  \left| \sum _{ u\in U, v\in V,} \chi_1(u) \right |\nonumber \\[.3cm]
	&\ll&  p^{1/2} \cdot \sqrt{\#U\cdot \#V }\sum _{x\leq p \leq 2x} 1\sum _{d_1\mid \varphi(p)}\left| \mu(d_1)\right |.
\end{eqnarray} 
The last step follows from $\pi(2x)-\pi(x)\ll x(\log x)^{-1}$ and $\sum _{d\mid n}\left| \mu(d)\right |\ll n^{\varepsilon}$.
\end{proof}
\begin{lem}   \label{lem2727SPRFR.550C}\hypertarget{lem2727SPRFR.550C}  Let $x>1$ be a large real number and let $\#U=z$ and $\#V=z$ be the cardinalities of the subsets of integers $U$ and $V$. If the functions $\chi_i\ne1$ are multiplicative characters of orders $d\mid \varphi(p)$ or $d\mid \varphi(p^2)$, then 
\begin{eqnarray} \label{eq2727SPRFR.400m4}
	E_1(x,z_0,z_1)
	&=&\sum _{ u\in U, v\in V,} \sum _{x\leq p \leq 2x} \left (\frac{\varphi(p-1)}{p-1}\sum _{\substack{d_1\mid \varphi(p)\\d_1>1}} \frac{\mu(d_1)}{\varphi(d_1)}\sum _{\ord \chi_1 =d_1} \chi_1(u+v) \right )   \nonumber   \\[.3cm]
	&&\hskip 1 in \times \left (\frac{\varphi(\varphi(p^2))}{p^2}\sum _{\substack{d_2\mid \varphi(p^2)\\d_2>1}} \frac{\mu(d_2)}{\varphi(d_2)}\sum _{\ord \chi_2 =d_2} \chi_2(u+v) \right ) \nonumber\\[.3cm]
 &\ll&  \frac{x\sqrt{z_0z_1}}{(\log x)(\log \log x)^2} \cdot p^{1/2+\varepsilon},
\end{eqnarray}
where $\varepsilon>0$ is a small real number.
\end{lem} 
\begin{proof}[\textbf{Proof}] Rearranging it and taking absolute values yield
	\begin{eqnarray} \label{eq2727SPRFR.400k}
		|E_1|
		&\leq& \sum _{x\leq p \leq 2x} \frac{\varphi(p-1)}{p-1}\frac{\varphi(\varphi(p^2))}{p^2}  \\[.3cm]
		&&\times \sum _{ u\in U, v\in V,} \left|\sum _{\substack{d_1\mid \varphi(p)\\d_1>1}} \frac{\mu(d_1)}{\varphi(d_1)}\sum _{\ord \chi_1 =d_1} \chi_1(u+v)\right| \left|\sum _{\substack{d_2\mid \varphi(p^2)\\d_2>1}} \frac{\mu(d_2)}{\varphi(d_2)}\sum _{\ord \chi_2 =d_2} \chi_2(u+v) \right|\nonumber.
	\end{eqnarray} 

An application of \hyperlink{lem2727EES.450E}{Lemma} \ref{lem2727EES.450E} and using the trivial bound yield
\begin{eqnarray} \label{eq2727SPRFR.400m}
|E_1|
&\ll& \sum _{x\leq p \leq 2x} \frac{\varphi(p-1)}{p-1}\frac{\varphi(\varphi(p^2))}{p^2}\cdot  \left(2^{\omega(p-1)}p^{1/2}\sqrt{\#U\cdot \#V }\right)  \left( 2^{\omega(p^2)}\right) \\[.3cm]
&\ll&  \left( \frac{x}{\log x)(\log \log x)^2}\right) \cdot 2^{\omega(p-1)+\omega(p^2)} \cdot p^{1/2}\cdot \sqrt{\#U\cdot \#V } \nonumber.
\end{eqnarray} 
Lastly, plug the estimate $2^{\omega(p-1)+\omega(p^2)}\ll p^{\varepsilon}$, see \eqref{eq4545EOFW400j}, and $z_0=\#U$ and $z_1=\#V$ to complete the verification.
\end{proof}

\section{The Least Stationary Primitive Roots } \label{S2727SPRFR-SPR}\hypertarget{S2727SPRFR-SPR}
The proof for the least stationary primitive root $g_s(p)$ modulo $p$ is essentially the same as that pioneered by Vinogradov and later authors. Excellent introductions to this technique appears in \cite{ES1957}, \cite{WR2001}, et cetera.

\begin{proof}[{\color{blue}\normalfont Proof of \hyperlink{thm2727SPRFR.050}{Theorem} \ref{thm2727SPRFR.050}}] Let $U,V\subset \Z$ be a pair of subsets of integers of cardinality $z_0=\#U$ and $z_1=\#V$ respectively. Let $g=u+v$ be a primitive root $\bmod p $, where $u\in U$ and $v\in V$. Suppose that there is no primes $p\in[x,2x]$ and there is no primitive root $g=u+v$ such that the primitive root $g\bmod p $ can be lifted to a primitive root $g\bmod p^2 $ --- equivalently, there is no stationary primitive roots $g_s(p)\leq z_0+z_1$. Summing the characteristic function over the domain $[x,2x]\times U\times V$ yields the nonexistence equation
	
	\begin{eqnarray} \label{eq2727SPRFR.050m2}
		N_2(x,z_0,z_1)&=&\frac{1}{2}\sum_{ u\in U, v\in V,} \sum_{x\leq p \leq 2x} \Psi_{p} (u+v)\left( 1+\Psi_{p^2} (u+v)\right)\\[.3cm]
		&=&\frac{1}{2}\sum_{ u\in U, v\in V,} \sum_{x\leq p \leq 2x} \Psi_{p} (u+v) 
		+\frac{1}{2}\sum_{ u\in U, v\in V,} \sum_{x\leq p \leq 2x}\Psi_{p} (u+v)\Psi_{p^2} (u+v)\nonumber\\[.3cm]
		&=&0\nonumber.
	\end{eqnarray}
	Replacing the characteristic functions, \hyperlink{lem2727SPRFR.100-C4}{Lemma} \ref{lem2727SPRFR.100-C4} into the nonexistence equation \eqref{eq2727SPRFR.050m2} and expanding it yield
	
	\begin{eqnarray} \label{eq2727SPRFR.050m4}
		0&=&\frac{1}{2}\sum_{ u\in U, v\in V,}  \sum_{x\leq p \leq 2x} \Psi_{p} (u+v)\Psi_{p^2} (u+v)  \\[.3cm]
		&=&\frac{1}{2}\sum_{ u\in U, v\in V,}  \sum_{x\leq p \leq 2x} \left (\frac{\varphi(p-1)}{p-1}
		\sum_{d_1\mid \varphi(p)} \frac{\mu(d_1)}{\varphi(d_1)}\sum _{\ord \chi_1 =d_1} \chi_1(u+v) \right )  \nonumber\\[.3cm]
		&&\hskip .5 in+\frac{1}{2}\sum_{ u\in U, v\in V,}  \sum_{x\leq p \leq 2x} \left (\frac{\varphi(p-1)}{p-1}\sum _{d_1\mid \varphi(p)} \frac{\mu(d_1)}{\varphi(d_1)}\sum_{\ord \chi_1 =d_1} \chi_1(u+v) \right )      \nonumber\\[.3cm]
		&&\hskip .5 in \times \left (\frac{\varphi(\varphi(p^2))}{p^2}\sum_{d_2\mid \varphi(p^2)} \frac{\mu(d_2)}{\varphi(d_2)}\sum _{\ord \chi_2 =d_2} \chi_2(u+v) \right ) \nonumber\\[.3cm]
		&=&M(x,z_0,z_1)\; +\; E(x,z_0,z_1),\nonumber
	\end{eqnarray} 
	where $M(x,z_0,z_1)$ is the main term specified by the trivial character $\chi=1$ and $E(x,z_0,z_1)$ is the error term specified by the nontrivial characters $\chi\ne$.\\

Set $\#U=\#V=z$ and rearranging the equation as $-E(x,z_0,z_1)=M(x,z_0,z_1)$. Next taking absolute value and applying \hyperlink{lem2727SPRFR.300-M}{Lemma} \ref{lem2727SPRFR.300-M} to the main term and \hyperlink{lem2727SPRFR.550-E}{Lemma} \ref{lem2727SPRFR.550-E} to the error term $E(x,z_0,z_1)$ yield
	\begin{eqnarray} \label{eq2727SPRFR.050m6}
		\left|-E(x,z_0,z_1)\right|&\geq &M(x,z_0,z_1)	\nonumber\\[.3cm]
\frac{x}{\log x} \cdot p^{1/2+\varepsilon}\cdot \sqrt{z_0  z_1 }&\gg&z_0 z_1\cdot\frac{x}{\log x}	\\[.3cm]
p^{1/2+\varepsilon}	&\gg&z	\nonumber.
\end{eqnarray} 
Setting $z=p^{1/2+\varepsilon}\log p$ contradicts the hypothesis \eqref{eq2727SPRFR.050m2} for all prime numbers $p \in [x,2x]$. Therefore, the least stationary  primitive root $g_s(p)\in \Z/p\Z$ satisfies the inequality 
\begin{equation}\label{eq2727SPRFR.050m8}
	g_s(p)\ll p^{1/2+\varepsilon}\log p
\end{equation} 
for all primes $p\in[x,2x]$ as $x\to \infty$.
\end{proof}

This technique is limited to small improvements as described in {\color{red}\cite[Section 4]{ES1957}}. However, an application of the Vinogradov and Burgess method, this can reduced to $g_s(p)\ll p^{1/4+\varepsilon}$, see \cite{CS1974}. \\

The explicit upper bound $h(p)<p^{0.99}$ for the least primitive root in $\Z/p^2\Z$ was proved in \cite{KT2019}, and recently improved to $h(p)<p^{2/3}$ in {\color{red}\cite[Theorem 4]{MS2022}}.\\

The actual upper bound $g_s(p)\leq h(p)$ of the least stationary primitive root is significantly smaller than \eqref{eq2727SPRFR.050m8}. In fact, the average value 
\begin{equation} \label{eq2727SPRFR.050ia}
	\pi(x)^{-1}\sum_{p\leq x}h(p)\ll (\log x)^{3+\varepsilon}, 
\end{equation}
of the least primitive root modulo $p^2$, which is computed in \cite{BD1971}, is closer to the actual value. Different techniques are needed to lower the upper bound \eqref{eq2727SPRFR.050m8} to the actual value.

\begin{conj}\label{conj2727SPRFR.050}\hypertarget{conj2727SPRFR.050} The least stationary primitive root satisfies the inequality
\begin{equation} \label{eq2727SPRFR.050ib}
	g_s(p)\ll (\log p)^{1+\varepsilon}, 
\end{equation} 
where $\varepsilon>0$ is a small number, as $x\to\infty$. 
\end{conj}

\section{Random Stationary Primitive Roots and Density of Primes} \label{S2727SPRFR-DP}\hypertarget{S2727SPRFR-DP}
The average density of primes $p\leq x$ with respect to a random primitive root $g(p)\in\Z/p\Z$, where $g\leq z$ and $z\gg e^{(\log x)^c}$, with $c>1/2$, was computed many decades ago in \cite{SP1969}. This section is concerned with the average density of prime numbers $p\in[x,2x]$ with respect to a random primitive root $g(p)\in\Z/p^k\Z$ for all $k\geq1$, where $g\leq z$ and $z\gg x^{c}$ with $c>1/2$. \\

Sine a simultaneous primitive root $g(p)\in\Z/p\Z$ and $g(p)\in\Z/p^2\Z$ is automatically a primitive root in $\Z/p^k\Z$ for all $k\geq1$, it is sufficient to prove the result for simultaneous primitive roots $g(p)\in\Z/p\Z$ and $g(p)\in\Z/p^2\Z$.

The corresponding counting function for the number of primes $p\in [x,2x]$ for which there is a primitive root $g\in [2,2z]$ modulo $p$, that can be lifted to a primitive roots modulo $p^2$, is defined by
\begin{eqnarray} \label{eq2727SPRFR.900i}
	N_2(x,z)	&=&\#\{ p \in [x,2x]: \ord_{p}g=\varphi(p),\;\ord_{p^2}g=\varphi(p^2) \text{ and } g\leq 2z\,\}\nonumber\\[.3cm]
	&=&	\#\{ \,\Psi_{s} (g):x\leq p\leq 2x \text{ and } g\leq 2z \,\},
\end{eqnarray}	
This is equivalent to the cardinality of the set of stationary primitive roots in $\Z/p\Z$, see the definition in \hyperlink{S2727SPRFR.200C}{Definition} \ref{dfn2727SPRFR.200C}.
The main result below answers an open question concerning the cardinality of the subset of integers \eqref{eq2727SPRFR.100d}. A discussion of this topic and the most recent numerical data appears in \cite{PA2009}.

\begin{thm} \label{thm2727SPRFR.900}\hypertarget{thm2727SPRFR.900} Let $x>1$ be a large real number, let $p\in[x,2x]$ be a prime number and let $z=p^{1/2+\varepsilon}\log p$. Then 
	\begin{equation} \label{eq2727SPRFR.900k}
		N_2(x,z)=c_2z^2\cdot\frac{x}{\log x}\left( 1+O\left( \frac{1}{\log x}  \right)  \right), 
	\end{equation} 
	where $c_2>0$ is a small constant,	as $x\to\infty$. 
\end{thm}

\begin{proof}[\textbf{Proof}] Let $U,V\subset \Z$ be a pair of subsets of integers of cardinalities $\#U=\#V=z$, and let $g=u+v$ be a primitive root $\bmod p $, where $u\in U$ and $v\in V$. Suppose that there is no primes $p\in[x,2x]$ and there is no primitive root $g=u+v$ such that the primitive root $g\bmod p $ can be lifted to a primitive root $g\bmod p^2 $. Summing the characteristic function over the domain $[x,2x]\times U\times V$ yields the nonexistence equation
	
	\begin{eqnarray} \label{eq2727SPRFR.900m2}
		N_2(x,z)&=&\frac{1}{2}\sum_{ u\in U, v\in V,} \sum_{x\leq p \leq 2x} \Psi_{p} (u+v)\left( 1+\Psi_{p^2} (u+v)\right)\\[.3cm]
		&=&\frac{1}{2}\sum_{ u\in U, v\in V,} \sum_{x\leq p \leq 2x} \Psi_{p} (u+v) 
		+\frac{1}{2}\sum_{ u\in U, v\in V,} \sum_{x\leq p \leq 2x}\Psi_{p} (u+v)\Psi_{p^2} (u+v)\nonumber\\[.3cm]
		&=&0\nonumber.
	\end{eqnarray}
	Replacing the characteristic functions, \hyperlink{lem2727SPRFR.100-C4}{Lemma} \ref{lem2727SPRFR.100-C4} into the nonexistence equation \eqref{eq2727SPRFR.900m2} and expanding it yield
	
	\begin{eqnarray} \label{eq2727SPRFR.900m4}
		N_2(x,z_0,z_1)&=&\frac{1}{2}\sum_{ u\in U, v\in V,}  \sum_{x\leq p \leq 2x} \Psi_{p} (u+v)\Psi_{p^2} (u+v)  \\[.3cm]
		&=&\frac{1}{2}\sum_{ u\in U, v\in V,}  \sum_{x\leq p \leq 2x} \left (\frac{\varphi(p-1)}{p-1}
		\sum_{d_1\mid \varphi(p)} \frac{\mu(d_1)}{\varphi(d_1)}\sum _{\ord \chi_1 =d_1} \chi_1(u+v) \right )  \nonumber\\[.3cm]
		&&\hskip .5 in+\frac{1}{2}\sum_{ u\in U, v\in V,}  \sum_{x\leq p \leq 2x} \left (\frac{\varphi(p-1)}{p-1}\sum _{d_1\mid \varphi(p)} \frac{\mu(d_1)}{\varphi(d_1)}\sum_{\ord \chi_1 =d_1} \chi_1(u+v) \right )      \nonumber\\[.3cm]
		&&\hskip .5 in \times \left (\frac{\varphi(\varphi(p^2))}{p^2}\sum_{d_2\mid \varphi(p^2)} \frac{\mu(d_2)}{\varphi(d_2)}\sum _{\ord \chi_2 =d_2} \chi_2(u+v) \right ) \nonumber\\[.3cm]
		&=&M(x,z_0,z_1)\; +\; E(x,z_0,z_1),\nonumber
	\end{eqnarray} 
	where $\#U=z_0$ and $\#V=z_1$.\\
	
	An asymptotic formula for the main term $M(x,z_0,z_1)$ is computed in \hyperlink{lem2727SPRFR.300-M}{Lemma} \ref{lem2727SPRFR.300-M} and an upper bound for the error term $E(x,z_0,z_1)$ is computed in \hyperlink{lem2727SPRFR.550-E}{Lemma} \ref{lem2727SPRFR.550-E}. Substituting these estimates yield
	\begin{eqnarray} \label{eq2727SPRFR.900m6}
		N_2(x,z_0,z_1)&=&M(x,z_0,z_1)\; +\; E(x,z_0,z_1)\\[.3cm]
		&=&c_2z_0 z_1\cdot\frac{x}{\log x}\left( 1+O\left(\frac{1 }{\log x}\right)  \right) \;+\;O\left( \frac{x}{\log x} \cdot p^{1/2+\varepsilon}\cdot \sqrt{z_0  z_1 }\right) 	\nonumber\\[.3cm]
		&=&c_2z_0 z_1\cdot\frac{x}{\log x} \left( 1\;+\;O\left(\frac{1}{\sqrt{z_0  z_1 }} \cdot p^{1/2+\varepsilon}\right)  \right) 	\nonumber.
	\end{eqnarray} 
	Setting $z=p^{1/2+\varepsilon}\log p$, $\#U=\#V=z$ and simplifying yield
	\begin{eqnarray} \label{eq2727SPRFR.900m8}
		N_2(x,z)&=&M(x,z)\; +\; E(x,z)\\[.3cm]
		&=&c_2z^2\cdot\frac{x}{\log x} \left( 1\;+\;O\left(\frac{1}{z} \cdot p^{1/2+\varepsilon}\right)  \right) 	\nonumber\\[.3cm].
		&=&c_2z^2\cdot\frac{x}{\log x}\left( 1+O\left( \frac{1}{\log x}  \right)  \right) 	\nonumber\\[.3cm]
		&>&0  \nonumber,
	\end{eqnarray}

	since $z=p^{1/2+\varepsilon}\log p>p^{1/2+\varepsilon}\geq x^{1/2+\varepsilon}$.	\\
	
	Clearly, the inequality \eqref{eq2727SPRFR.900m8} contradicts the hypothesis \eqref{eq2727SPRFR.900m2} for all sufficiently large prime numbers $p \in [x,2x]$. Therefore, on average, a random primitive root $g\in \Z/p\Z$ such that $g\leq 2p^{1/2+2\varepsilon}$ is stationary all the finite rings $\Z/p^k\Z$ for a positive proportion $c_2>0$ of the primes $p\in[x,2x]$ and $k\geq1$. 
\end{proof}

A numerical estimate of the density 
\begin{equation} \label{eq2727SPRFR.900m9}
	c_2<0.26065286200344619944228095665445442967965\ldots
\end{equation}
is outlined in \eqref{eq4040DP.200k2}.

\section{Problems } \label{S2356QRP}
\subsection{Rational expansions problems}
\begin{exe} \label{exe2356QRP.110} {\normalfont  Let $a\ne\pm1, v^2$ be an integer. Prove a (a variant of the Zeremba conjecture) that the continued fraction of $1/p$ is of the form $1/p=[0;a_1,a_2,a_3,\ldots]$, where the partial quotients $a_i\in [0,a-1]$, holds for a positive proportion of the primes $p\in\tP=\{2,3,5,\ldots\}$.
	} 
\end{exe}
\vskip .15 in 
\begin{exe} \label{exe2356QRP.110d} {\normalfont  Prove or disprove that the Erdos-Straus type conjecture for Egyptian fractions $$\frac{4}{p}=\frac{1}{a}+\frac{1}{b}+\frac{1}{c},$$
		where $a,b,c\in\N=\{1,2,3,\ldots\}$, holds for a positive proportion of the primes $p\in\tP=\{2,3,5,\ldots\}$.
	} 
\end{exe}
\vskip .25 in 
\begin{exe} \label{exe2356QRP.110j} {\normalfont  Let $a\ne \pm1,v^2$ be a fixed integer. Prove or disprove that the twisted reciprocal order series $$\sum _{p\geq2}\frac{\mu(p-1)}{p^s\cdot \ord_p a}$$
converges in the complex half plane $\{s=\sigma+it:\Re es\geq 1\}\subset \C$.
	} 
\end{exe}
\vskip .25 in 
\subsection{Prime divisor function problems}
\begin{exe} \label{exe2356QRP.010d} {\normalfont  Let $x\geq 1$ be a large number. Let $\omega(n)=\#\{p\mid n:p\text{ is prime }\}$ and let $\ord_p a$ be the multiplicative order of $a\ne\pm1,v^2$ modulo $p$. Estimate the finite sums 
		$$\sum _{\substack{p\leq x\\\ord_p a=(p-1)/2}}\omega(p-1) \quad\quad \text{ and }\quad\quad\sum _{\substack{p\leq x\\\ord_p a=p-1}}\omega(p-1) .$$
	} 
\end{exe}
\vskip .15 in
\begin{exe} \label{exe2356QRP.010k} {\normalfont  Let $x\geq 1$ be a large number. Let $\omega(n)=\#\{p\mid n:p\text{ is prime }\}$ and let $\mu(n)$ be the Mobius function respectively. Estimate the finite sums
		$$\sum _{\substack{p\leq x\\\mu(p-1)=-1}}\omega(p-1) \quad\quad \text{ and }\quad\quad\sum _{\substack{p\leq x\\\mu(p-1)=1}}\omega(p-1) .$$
	} 
\end{exe}
\vskip .15 in
\begin{exe} \label{exe2356QRP.010j} {\normalfont  Let $x\geq 1$ be a large number, let $q=(\log x)^c$ and let $\omega(n)=\#\{p\mid n:p\text{ is prime }\}$ be the prime divisor counting function. Find an asymptotic formula for 
		$$\sum _{\substack{p\leq x\\p\equiv a \bmod q}}\omega(p-1)\overset{?}{=}a_1\frac{x}{\varphi(q)\log x}+O\left( \frac{x}{(\log x)^b}\right) ,$$
		where $\gcd(a,q)=1$, $a_1=0.37\ldots$ and $b>c+1\geq1$ are constants,
	} 
\end{exe}
\vskip .15 in 

\begin{exe} \label{exe2356QRP.010} {\normalfont  Let $x\geq 1$ be a large number and let $\omega(n)=\#\{p\mid n:p\text{ is prime }\}$ be the prime divisor counting function.  Determine whether or not
		$$\sum _{p\leq x}	2^{\omega(p-1)}\overset{?}{=}c_0\frac{x}{\log x}\log \log x+c_1\frac{x}{\log x}+O\left( \frac{x}{(\log x)^2}\right) ,$$
		where $c_0,c_1>0$ are constants,
	} 
\end{exe}
\vskip .15 in
\begin{exe} \label{exe2356QRP.110f} {\normalfont  Find an estimate for the twisted finite sum $$\sum _{p\leq x}	\mu(p-1)\omega(p-1)$$
as $x\to \infty.$	} 
\end{exe}
\vskip .25 in 
\vskip .15 in 


\end{document}